\documentclass[10pt]{amsart}
\usepackage{amscd, amsfonts, amssymb}
\usepackage{epsfig}
\usepackage[mathscr]{eucal}
\usepackage{verbatim}
\usepackage{latexsym}
\usepackage{lscape}
\usepackage{enumerate}

\DeclareMathOperator{\GL}{GL}
\DeclareMathOperator{\IM}{Im}
\DeclareMathOperator{\spp}{supp}

\newcommand\sse{\subseteq}
\newcommand{\ovl}[1]{{\overline{#1}}}
\newcommand{\g}{\gamma}

\numberwithin{equation}{section}

\newtheorem{thm}[equation]{Theorem}
\newtheorem{lem}[equation]{Lemma}
\newtheorem{cor}[equation]{Corollary}
\newtheorem{prop}[equation]{Proposition}

\theoremstyle{definition}
\newtheorem{defn}[equation]{Definition}
\theoremstyle{remark}

\newtheorem*{Rem}{Remark}

\newtheorem*{Rems}{Remarks}

\thanks{2000 {\it Mathematics Subject Classification}.
20G15, 14L24.}
\keywords{reductive group, $G$-module, Jordan decomposition, orbit closure, rationality}

\title[Rationality and the Jordan decomposition]
{Rationality and the Jordan-Gatti-Viniberghi decomposition}

\author[J.\ Levy]{Jason Levy}
\address
{Department of Mathematics and Statistics
University of Ottawa
585 King Edward
Ottawa, ON K1N 6N5}
\email{jlevy@uottawa.ca}

\begin{document}

\begin{abstract}
 We verify the conjecture of \cite{me} and use it to prove that the
 semisimple parts of the rational Jordan-Kac-Vinberg decompositions of
 a rational vector all lie in a single rational orbit.
\end{abstract}

\maketitle

\section{Introduction}
\label{sec:intro}

Let $k$ be a field of characteristic 0, and write $\ovl k$ for its algebraic closure.  Let $G$ be a reductive algebraic group (not
necessarily connected), acting on a vector space $V$, with $G$, $V$,
and the action all defined over $k$.  Given a point $v\in V$, write $G_v$ for the stabilizer of $v$; it is an algebraic subgroup of $G$.

In \cite{Gatti}, Kac and Vinberg made the following definitions:

\begin{defn} (a) A vector $s\in V$ is {\it semisimple} if the orbit
  $G\cdot s$ is Zariski closed.\\   
(b) A vector $n\in V$ is {\it nilpotent} with respect to $G$ if the
Zariski closure $\ovl{G\cdot n}$ contains the vector 0.\\ 
(c) A {\it Jordan decomposition} of a vector $\g\in V$ is a decomposition $\g =s + n$, with
\begin{enumerate}
\item $s$ semisimple,
\item $n$ nilpotent with respect to $G_s$,
\item $G_\g \sse G_s$.
\end{enumerate}
\end{defn}

This Jordan-Kac-Vinberg decomposition matches the standard Jordan decomposition when $V$ is the 
Lie algebra $\mathfrak g$ of $G$.  In that case, every $\g\in\mathfrak g$ has a unique Jordan
decomposition $\g=s+n$, and if $\g$ lies in $\mathfrak g(k)$ then so do $s$ and $n$.   For
general $V$, however, as noted in \cite{kac_proof}, an element $\g\in V$ may
have multiple Jordan-Kac-Vinberg decompositions.  For all of them, the element $s$
lies in a single $G$-orbit, namely the unique closed $G$-orbit in
$\ovl{G\cdot \g}$.

In \cite{kac_proof}, Kac showed that a simple application of the Luna
slice theorem shows that every vector $\g\in V$ has a
Jordan-Gatti-Viniberghi decomposition (as he called it).  A rational
version of the Luna slice theorem has been proven by Bremigan
\cite{Brem}, and it implies (Lemma~\ref{k-Jordan}) that every $k$-point $\g\in V(k)$
has a $k$-Jordan-Kac-Vinberg decomposition, that is a Jordan-Kac-Vinberg decompostion
$\g=s+n$ with $s$ (and hence $n$) in $V(k)$.  This fact has not previously appeared
in the literature.

In this paper we show that given $\g\in V(k)$, the semisimple parts $S$ of all $k$-Jordan-Kac-Vinberg
decompositions $\g=s+n$ of $\g$ all lie in a single $G(k)$-orbit. In other words, even though
$k$-Jordan-Kac-Vinberg decompositions are not unique, the $G(k)$-orbit of the semisimple parts is. This
uniqueness is important in producing the fine geometric expansion in relative trace formulas 
(see the discussion in \cite{me_Pois}).

We use two tools to prove this.  One is the rational version of the
Hilbert-Mumford Theorem, as proven by Kempf \cite{kempf} and Rousseau.
The Hilbert-Mumford Theorem allows us to restate the problem in terms
of limits of the form $\lim_{t\to 0} \lambda(t)\cdot \g$, for $k$-cocharacters $\lambda$ in $G$.

The other is a recent rationality result of Bate-Martin-R\"ohrle-Tange
\cite{BMRT} on such limits.

In fact we prove the somewhat more general result,  Theorem \ref{thm}, that
given $\g\in V(k)$, the limit points $\lim_{t\to 0}
\lambda(t)\cdot \g$ that are semisimple, ranging over all
cocharacters $\lambda$ defined over $k$, all lie in a single
$G(k)$-orbit.  This solves the conjecture in~\cite{me}.  

\section{Preliminaries}

We begin with some notation.  Let $k$ be a field (of any characteristic), and write
$\ovl k$ for its algebraic closure.
Let $G$ be a reductive algebraic group (not
necessarily connected) defined over $k$. Write $X^*(G)$ for the group of
characters $\chi\colon G \to \GL(1)$, and $X_*(G)$ for the set of
cocharacters $\lambda\colon \GL(1) \to G$.  Similarly write $X^*(G)_k$
(resp. $X_*(G)_k$) for those characters (resp. cocharacters)
defined over $k$.  Define the map  
$\langle \, , \, \rangle \colon X_*(G) \times X^*(G) \to \mathbb Z$
by requiring the identity 
$\chi(\lambda(t))=t^{\langle \lambda, \chi   \rangle}$.  The group $G$
acts naturally on $X_*(G)$:
\[
(g\cdot \lambda)(t) = g\lambda(t)g^{-1}, \text{ for } g\in G,\
\lambda\in X_*(G),\  t\in \GL(1).
\]
Given $\lambda\in X_*(G)_k$ and $g\in G(k)$, the cocharacter 
$g\cdot \lambda$ is also in $X_*(G)_k$.

Suppose that $V$ is an affine $G$-variety.  Given $\lambda\in
X_*(G)$ and $v\in V$, we say that the {\it limit}
\begin{equation}\label{lim}
\lim_{t\to 0} \lambda(t)\cdot v
\end{equation}
exists and equals $x$ if there is a morphism of varieties 
$\ell\colon \mathbb A^1 \to V$ with $\ell(t) = \lambda(t)\cdot v$ for
$t\ne 0$, and $\lambda(0)=x$.  Notice that if $\ell$ exists then it is unique; also
if $V$ and $\lambda$ are defined over $k$, with $v\in V(k)$, then $\ell$ must
also be defined over $k$, and so $x$ must lie in $V(k)$.
Given $v\in V(k)$, write $\Lambda(v,k)$ for the set of $\lambda\in
X_*(G)_k$ such that the limit
(\ref{lim}) exists.

The group $G$ acts on itself via the action $y\mapsto xyx^{-1}$.
Given $\lambda\in X_*(G)$, let $P(\lambda)$ be the subvariety  
$$P(\lambda)=
\{\, g\in G\mid \lim_{t\to 0} \lambda(t) g \lambda(t)^{-1} \text{
  exists}\,\}.$$
It is an algebraic group, defined over $k$ if $\lambda$ is.  These
groups $P(\lambda)$ were defined in \cite{tuples} and are called the
Richardson parabolic subgroups in \cite{BMRT}.  The map 
$$
h_\lambda\colon P(\lambda) \to G,\quad 
h_\lambda(g)= \lim_{t\to 0} \lambda(t) g \lambda(t)^{-1}
$$
is a homomorphism of algebraic groups, defined over $k$ if $\lambda$
is.  The image and kernel are given by
\begin{align*}
\IM h_\lambda &= G^\lambda = \{\, g\in G\mid \lambda(t)g\lambda(t)^{-1}=g, \text{ for
  all }t\,\}\\
\ker h_\lambda &= 
\{\, g\in G\mid \lim_{t\to 0}\lambda(t)g\lambda(t)^{-1}=1\,\}
= R_u(P(\lambda))
\end{align*}
(see \cite{tuples} for details).

Now suppose that $V$ is a $G$-module defined over $k$.  Given any
$\lambda\in X_*(G)$, we can then define the $G^\lambda$-modules
\begin{align*}
V_{\lambda, n} &= \{ v\in V \mid \lambda(t)\cdot v = t^n v \text{ for
  all } t\,\}, \ n\in \mathbb Z\\
V_{\lambda, +} &=\sum_{n>0}V_{\lambda,n},\quad
V_{\lambda,0+}=\sum_{n\ge0}V_{\lambda,n}=V_{\lambda,0}\oplus V_{\lambda,+}.
\end{align*}
Notice that $V_{\lambda,0+}$ consists of those vectors $v\in V$ such that the
limit (\ref{lim}) exists, and is invariant under $P(\lambda)$; in fact
for $g\in P(\lambda)$, and $v\in V_{\lambda,0+}$,
\begin{equation}\label{p-conj}
\lim_{t\to 0}\lambda(t)\cdot (g\cdot v) = h_\lambda(g)\cdot
\left(\lim_{t\to 0}\lambda(t)\cdot v\right).
\end{equation}
Further, for $v\in V_{\lambda,0+}=V_{\lambda,0}\oplus V_{\lambda,+}$,
the limit (\ref{lim}) is just the projection of $v$ to
$V_{\lambda,0}$.

Suppose next that $A$ is a maximal $k$-split torus in $G$.  For each
$\chi\in X^*(A)_k$ define $V^\chi$ by
\[
V^\chi = \{ v\in V\mid a\cdot v = \chi(a)v, \quad \text{for all }a\in A\}.
\]
Then only finitely many $V^\chi$ are nonzero and $V$ is their direct
sum.  Given a vector $v\in V$, write $v_\chi$ for the component of $v$ in
the space $V^\chi$, $\chi\in X^*(A)_k$, and set
\[
\spp v = \spp_A v = \{ \chi\in X^*(A)_k\mid v_\chi \ne 0\},
\]
so that
\[
v=\sum_{\chi\in \spp v} v_\chi.
\]

For any $\lambda\in X_*(A)_k \subset X_*(G)_k$, each vector space
$V_{\lambda,n}$, $n\in \mathbb Z$, is also a direct sum of weight spaces:
\[
V_{\lambda,n} = 
\sum_{{\chi\in X^*(A)_k}\atop {\langle \lambda,\chi\rangle=n}}
V^\chi.
\]
We record the following obvious statement for later use.
\begin{lem}\label{limits}
For a vector $v\in V$, the limit (\ref{lim}) exists if and only if for
every $\chi\in \spp v$ we have $\langle \lambda, \chi \rangle \ge 0$; in
this case the limit equals
\[
\sum_{{\chi\in X^*(A)_k}\atop {\langle \lambda,\chi\rangle=0}} v_\chi.
\]
\end{lem}

\section{Limits}

In this section we assume that $k$ is perfect.  We summarize some
results that we will later use. First is the rational version of the
Hilbert-Mumford Theorem, 
\cite[Cor.~4.3]{kempf}. 
\begin{lem}\label{HM}
If $\g\in V(k)$, then there exists $\lambda\in X_*(G)_k$ so that the limit
$$\lim_{t\to 0} \lambda(t)\cdot \g$$
exists and is semisimple.
\end{lem}
\begin{Rem}Note that this limit point must necessarily lie in $V(k)$.
\end{Rem}
The following two results are also essential to our proof.  The first
is a restatement of \cite[Lemma 2.15]{BMRT}.
\begin{lem}\label{2.15} Suppose that $A\subset G$ is a torus and $\lambda,
  \lambda_0$ are in $X_*(A)_k$. Suppose that vectors $\g, v_0, v'\in V$ are
related by 
\begin{align*}
v_0&=\lim_{t\to 0} \lambda_0(t)\cdot \g,\\
v' &=\lim_{t\to 0} \lambda(t)\cdot v_0.
\end{align*}
Then there exists $\mu\in X_*(A)_k$ such that
\begin{align*}
V_{\mu,0} &= V_{\lambda_0, 0} \cap V_{\lambda,0}\\
V_{\mu,+} &\supseteq V_{\lambda_0,+}, \quad
V_{\mu,0+} \subseteq V_{\lambda_0,0+},\\
v' &= \lim_{t\to 0} \mu(t)\cdot \g.
\end{align*}
\end{lem}
\begin{Rem}
The cocharacter $\mu$ can be of the form $n\lambda_0 + \lambda$ for
any sufficiently large $n\in\mathbb N$.
\end{Rem}

The second result is \cite[Cor.~3.7]{BMRT}.
\begin{lem} \label{3.7}
Let $v\in V(k)$ be semisimple.  For every
  $\lambda\in X_*(G)_k$, if the limit $\lim_{t\to 0} \lambda(t)\cdot v$
  exists, then it lies in $G(k)\cdot v$.
\end{lem}
\begin{Rem}
In fact \cite[Cor.~3.7]{BMRT} shows that the limit must lie in
$R_u(P(\lambda))(k)\cdot v.$
\end{Rem}

Our main result in this section is the following:
\begin{thm} \label{thm} Let $G$ be a reductive group and $V$ a $G$-module.
  Suppose that $k$ is perfect and let $\g\in V(k)$.  Then for every
  $\lambda, \mu\in X_*(G)_k$ such that both vectors 
  $v=\lim_{t\to 0}\lambda(t)\cdot \g$ and $v'=\lim_{t\to 0}\mu(t)\cdot \g$
  exist and are semisimple, $v'$ lies in $G(k)\cdot v$.
\end{thm}

\begin{Rems} (a) This solves Conjecture 1.5 of \cite{me}.\\
(b) As is well-known (see for example \cite[Remark 2.8]{BMRT} or
\cite[Lemma 1.1]{kempf}), we can embed any affine
$G$-variety over $k$ inside a $k$-defined rational $G$-module, and
hence Theorem \ref{thm} is also valid for affine $G$-varieties.
\end{Rems}

\begin{defn} Let $\Lambda(\g,k)_{\rm min}$ be the set of cocharacters
  that minimize $\dim V_{\lambda, 0}$, among $\lambda\in X_*(G)_k$
  such that $\lim_{t\to 0} \lambda(t)\cdot \g$ exists and is semisimple.
\end{defn}
\begin{Rem}By the Kempf-Rousseau-Hilbert-Mumford theorem \ref{HM}, and because $\dim V_{\lambda, 0}$ is always a nonnegative integer, the set $\Lambda(\g,k)_{\rm min}$ is non-empty.
\end{Rem}

\begin{lem} \label{uv}
Given $\lambda\in  \Lambda(\g,k)_{\rm min}$ and $p\in P(\lambda)(k)$, we have that $\lambda \in \Lambda(p\cdot \g, k)_{\rm min}.$  Further, the limit points
$$
v=\lim_{t\to 0}\lambda(t)\cdot \g \quad\text{and }\quad \lim_{t\to 0}\lambda(t) \cdot (p\cdot \g)
$$
lie in the same $G(k)$-orbit.
\end{lem}
\begin{proof} Let $\lambda\in  \Lambda(\g,k)_{\rm min}$.  By (\ref{p-conj}),
$\lim_{t\to 0} \lambda(t) \cdot (p\cdot \g) =h_\lambda(p) \cdot v ,$
so the limit exists and lies in the $G(k)$-orbit of $v$; consequently its $G$-orbit is closed.

On the other hand, given any 
$\mu\in \Lambda(p\cdot \g,k)$, we have that $p^{-1}\cdot \mu \in \Lambda(\g,k)$
and
$
\dim V_{\mu,0} = \dim V_{p^{-1}\cdot \mu, 0};
$ 
since $\lambda\in  \Lambda(\g,k)_{\rm min}$, this dimension is at
least $\dim V_{\lambda,0}$; hence $\lambda$ lies in
$\Lambda(p\cdot g, k)_{\rm min}$
\end{proof}

\begin{lem} \label{commuting}
Given $\lambda_0\in \Lambda(\g,k)_{\rm min}$, write $v_0=\lim_{t\to
  0}\lambda_0(t)\cdot \g$.  Suppose that $A\subset G$ is a torus with
$\lambda_0\in X_*(A)_k$.  Suppose that for $\lambda\in X_*(A)_k$,
the limit $v=\lim_{t\to 0} \lambda(t)\cdot \g$ exists and has a closed
$G$-orbit.  Then $v$ lies in $G(k)\cdot v_0$. 
\end{lem}
\begin{proof} By Lemma \ref{limits} the existence
  of the limit $v$ implies that for every $\chi\in \spp(\g)$
  we have $\langle \lambda, \chi\rangle \ge 0$, and the vector $v$ is
  the sum 
$$\sum_{\chi\in\spp \g\atop \langle \lambda ,\chi\rangle =0}\g_\chi,$$
the projection of $\g$ to $V_{\lambda,0}$;
in particular 
$\spp v \sse \spp \g$ and $\g-v\in V_{\lambda,+}$
.  Similarly $\spp(v_0)$ is 
contained in $\spp(\g)$, and so by Lemma \ref{limits} we may conclude
that the limit
$$v'=\lim_{t\to 0} \lambda(t)\cdot v_0$$
exists.  Since $G\cdot v_0$ is closed, $v'$ lies in $G\cdot v_0$, so that $G\cdot v'$ is also closed.

We then obtain, from Lemma \ref{2.15},
a $\mu\in \Lambda(\g,k)$ with $v'=\lim_{t\to 0} \mu(t)\cdot \g$, 
having a closed $G$-orbit, and
\begin{align}
\label{eq:v0}V_{\mu,0} &= V_{\lambda_0, 0} \cap V_{\lambda,0}\\
\label{eq:v+}V_{\mu,+} &\supseteq V_{\lambda_0,+}, \quad
V_{\mu,0+} \subseteq V_{\lambda_0,0+},
\end{align}

Since $\lambda_0$ lies in $\Lambda(\g,k)_{\rm min}$, we may conclude that $V_{\mu, 0}=V_{\lambda_0,0}$, and hence by (\ref{eq:v0}), (\ref{eq:v+}), also that
$V_{\mu,0+} = V_{\lambda_0,0+}$.  The limit point $v'$ is the
projection of $\g$ to $V_{\mu,0} = V_{\lambda_0,0}$, hence $v'=v_0$.

Since $\lim_{t\to 0} \mu(t)\cdot \g$ exists, $\g$ and hence $v$ lie in $V_{\mu,0+}$.
Now, the projection of $\g-v$ to 
$$V_{\lambda,0}
=\sum_{{\chi\in X^*(A)_k}\atop{\langle \lambda, \chi \rangle = 0}}
V^\chi$$
is zero.
By (\ref{eq:v0}), $V_{\mu,0}\subseteq V_{\lambda,0}$, so the
projection of $\g-v\in V_{\mu,0+}$ to $V_{\mu,0}$ is also zero, 
and hence
$$
\lim_{t\to 0}\mu(t)\cdot v =\lim_{t\to 0} \mu(t)\cdot \g = v' = v_0.
$$
By \ref{3.7}, we can finally conclude that $v$ lies in $G(k)\cdot v_0$.
\end{proof}

\begin{proof}[Proof of Theorem \ref{thm}]
First, note that a cocharacter in $G$ is necessarily a cocharacter in
the connected component $G^0$ of the identity in $G$, and that it is
sufficent to prove Theorem \ref{thm} for $G^0$.  Without loss of
generality, we therefore assume that $G$ is connected.

Pick $\lambda_0\in\Lambda(\g,k)_{\rm min}$, set $v_0 = \lim_{t\to 0}
\lambda_0(t)\cdot \g$.  Since being in the same $G(k)$-orbit is an
equivalence relation, it is clearly sufficient to prove the theorem
for $\mu=\lambda_0$, $v'=v_0$.

The image of $\lambda_0$ lies in a maximal torus, and by
\cite[1.4]{Borel-Tits} must in fact lie in a maximal $k$-split torus $A$.
Fix a minimal $k$-defined parabolic subgroup $P$ of $G$, with
$C_G(A)\subseteq P \subseteq P(\lambda_0)$.  The choice of $P$
corresponds to a choice of basis ${}_k\Delta$ of 
simple roots of $G$ with respect to $A$.

The image of $\lambda$ also lies in some maximal $k$-split torus, so since
all maximal $k$-split tori are conjugate over $G(k)$
(\cite[Thm.~20.9(ii)]{Bo}), there exists $g\in G(k)$ so that the
image of $g \cdot \lambda$ lies in $A$.  Multiplying $g$ on the
left by an element of $N_G(A)(k)$ if necessary, we can arrange that
$\langle g\cdot \lambda, \alpha\rangle \ge 0$ for every $\alpha
\in {}_k\Delta$, that is, $P \subseteq P(g\cdot\lambda)$.  Let
us write $\lambda_A$ for $g\cdot \lambda \in X_*(A)$. 

We now apply the Bruhat decomposition: write 
$$
g=p w u, \quad p\in P(k)\subseteq P(\lambda_A)(k),\ 
w\in N_G(A), \ u\in R_u(P)(k).
$$
Then 
\begin{align}
v & = \lim_{t\to 0} \lambda(t)\cdot \g = g^{-1}\cdot \lim_{t\to 0}
   \lambda_A(t) g\cdot \g \\ 
\label{eq:line 2}&= g^{-1}\cdot \left[ \lim_{t\to 0}
  \lambda_A(t)p\lambda(t)^{-1}\right] \cdot \lim_{t\to 0}
\lambda_A(t)  w u \cdot \g\\ 
&=g^{-1} h_{\lambda_A}(p) \cdot \lim_{t\to 0} 
   \lambda_A(t)   w u \cdot \g\\ 
& = g^{-1} h_{\lambda_A}(p)  w \cdot \lim_{t\to 0} 
   ( w^{-1}\cdot \lambda_A)(t)\cdot (u\cdot \g),
\end{align}
with $g h_{\lambda_A}(p) w\in G(k)$.  Note that the existence of the
first limit in (\ref{eq:line 2}) implies the existence of the second.

Now, $u\in R_u(P)(k) \subseteq P(k) \subseteq P(\lambda_0)(k)$, so by
Lemma \ref{uv}, 
$\lambda_0 \in \Lambda(u\cdot \g, k)_{\rm min}.$  Notice also that
$\lambda_0$ and $ w \cdot \lambda_A$ both lie in $X_*(A)_k$.  By
Lemmas \ref{commuting} and \ref{uv},
$$
\lim_{t\to 0} ( w \cdot \lambda_A) \cdot (u\cdot \g) \in 
G(k) \cdot \lim_{a\to 0} \lambda_0(t)\cdot(u\cdot \g) = G(k)\cdot v_0
$$
so $v$ is also in $G(k)\cdot v_0$.
\end{proof}

\section{Application to Jordan decompositions}

In this section we require $k$ to have characteristic 0.

\begin{defn} (a) A {\it Jordan-Kac-Vinberg decomposition} of a vector $\g\in V$ is a
  decomposition $\g =s + n$, with 
\begin{enumerate}
\item $s$ semisimple,
\item $n$ nilpotent with respect to $G_s$,
\item $G_\g \sse G_s$.
\end{enumerate}
(b) Given $\g\in V(k)$, a {\it $k$-Jordan-Kac-Vinberg decomposition} of $\g$ is a
Jordan-Kac-Vinberg decomposition $\g=s+n$ with $s$ (and hence $n$) in $V(k)$.
\end{defn}

Kac \cite{kac_proof} used the Luna Slice theorem to prove that every
vector has a Jordan-Kac-Vinberg decomposition. We now show that every
vector in $V(k)$ has a $k$-Jordan-Kac-Vinberg decomposition.

Bremigan proved a rational version of the Luna Slice Theorem in
\cite{Brem}.  The following, \cite[Cor.~3.4]{Brem} is an immediate
consequence of it.
\begin{lem} \label{Brem}
Given $v\in V$ semisimple, let $F$ be the set of points $x\in V$ with
$G\cdot v \sse \ovl{G\cdot x}$. Then there is a
$G$-invariant retraction $\psi\colon F \to G\cdot v$ that is defined over $k$.
\end{lem}
\begin{Rem}For fields of positive characteristic, the Luna Slice
  Theorem does not hold without additional assumptions.  See
  \cite{Bardsley-Richardson} for further details.
\end{Rem}

\begin{cor}\label{k-Jordan}Every $\g\in V(k)$ has a $k$-Jordan-Kac-Vinberg decomposition.\end{cor}
\begin{proof} 
Let $\g\in V(k)$.  By Lemma \ref{HM}, there exists a
semisimple $v\in \ovl{G\cdot \g} \cap V(k).$   Lemma \ref{Brem}
provides a $G$-invariant map $\psi$, defined over $k$, from 
$$F=\{\,x\in V\mid  \ovl{G\cdot x} \supseteq G\cdot v\,\}$$
to $G\cdot v$. Setting $s=\psi(\g)$, we immediately see that $s\in
V(k)$, that $s$ is semisimple, and that $G_\g\sse G_s$.  Furthermore,
since $\psi$ is $G$-invariant, $\psi(G_s\cdot\g)=s$, and hence 
$\psi(\ovl{G_s\cdot \g})= s$, so that the unique closed $G_s$-orbit in 
$\ovl{G_s\cdot\g}$ is $s$. Subtracting $s$, the unique closed
$G_s$-orbit in $\ovl{G_s\cdot(\g-s)}$ is 0.  Therefore $\g=s+(\g-s)$
is a $k$-Jordan-Kac-Vinberg decomposition.
\end{proof}

We can use the Hilbert-Mumford theorem to provide an alternate
description of a Jordan-Kac-Vinberg decomposition.

\begin{prop}\label{Jordan-lambda}
A decomposition $\g=s+n$, with $s$ semisimple, and $G_\g\sse G_s$, is a Jordan-Kac-Vinberg
decomposition if and only if there exists $\lambda\in X_*(G_s)$ so that
\begin{equation}\label{to_s}
\lim_{t\to 0}\lambda(t)\cdot \g = s.
\end{equation}
If $\g\in V(k)$, then
$\g=s+n$ is a $k$-Jordan-Kac-Vinberg decomposition if and only if $\lambda$ can be
taken to be in $X_*(G_s)_k$.
\end{prop}
\begin{proof}The first part of the Proposition is just the second part
  over $\ovl k$, so we need only consider the second part.  

Given a $k$-Jordan-Kac-Vinberg decomposition $\g=s+n$, we know that $0\in
\ovl{G_s\cdot n}$.  The Hilbert-Mumford Theorem (Lemma \ref{HM})
provides a $\lambda \in X_*(G_s)_k$ such that 
\begin{equation}\label{to_0}
\lim_{t\to 0}\lambda(t)\cdot n = 0.
\end{equation}
However, since the image of $\lambda$ is in $G_s$, we can add $s$ and
obtain (\ref{to_s}).

In the other direction, given $\lambda\in X_*(G_s)_k$, subtracting $s$
from (\ref{to_s}) gives (\ref{to_0}), implying that $n$ is nilpotent with
respect to $G_s$.
Since $\g$ and $\lambda$ are defined over $k$, so are $s$ and $n$,
hence $\g=s+n$ is a $k$-Jordan-Kac-Vinberg decomposition. 
\end{proof}

From Proposition~\ref{Jordan-lambda} and Theorem~\ref{thm}, we
immediately obtain the following:
\begin{cor}For any two $k$-Jordan-Kac-Vinberg decompositions $\g=s+n$, $\g=s'+n'$ of
  $\g\in V(k)$, we have $s'\in G(k)\cdot s$.
\end{cor}

This means that although a vector $\g\in V(k)$ may have multiple
$k$-Jordan-Kac-Vinberg decompositions, all such decompositions lie in a single
$G(k)$-orbit.

\proof[Acknowledgements]
We thank G. R\"ohrle for pointing out that the proof of Theorem \ref{thm}
applies, and hence Theorem \ref{thm} also holds, for any perfect field
$k$; and also for his careful proofreading.

\bibliographystyle{plain}

\end{document}